\documentclass[fleqn]{article}
\usepackage{amsmath}

\usepackage{amssymb}
\usepackage{epsfig}
\usepackage{graphicx}
\usepackage{tikz}
\usepackage{bm}
\usepackage{authblk}
\usepackage{IEEEtrantools}

\numberwithin{equation}{section}

\newtheorem{theorem}{Theorem}[section]
\newtheorem{lemma}[theorem]{Lemma}
\newtheorem{exa}[theorem]{Example}
\newtheorem{corollary}[theorem]{Corollary}
\newtheorem{proposition}[theorem]{Proposition}
\newtheorem{definition}[theorem]{Definition}

\newtheorem{remark}[theorem]{Remark}

\newenvironment{proof}{{\bf Proof.}}{}

\title{Contact join-semilattices}

\author{Tatyana Ivanova}
\affil{Institute of Mathematics and Informatics\authorcr
		Bulgarian Academy of Sciences\authorcr
		e-mail: tatyana.ivanova@math.bas.bg}
\date{}

\begin{document}

\maketitle

\begin{abstract}
Contact algebra is one of the main tools in region-based theory of space. In \cite{dmvw1, dmvw2,iv,i1} it is generalized by dropping the operation Boolean complement. Furthermore we can generalize contact algebra by dropping also the operation meet. Thus we obtain structures, called contact join-semilattices (CJS) and structures, called distributive contact join-semilattices (DCJS). We obtain a set-theoretical representation theorem for CJS and a relational representation theorem for DCJS. As corollaries we get also topological representation theorems. We prove that the universal theory of CJS and of DCJS is the same and is decidable.
\end{abstract}

\section{Introduction}

In classical Euclidean geometry the notion of point is taken as one of the basic primitive notions. In contrast, region-based theory of space (RBTS) has
as primitives the more realistic notion of \textit{region} (abstraction of physical body) together with some
basic relations and operations on regions. Some of these relations are mereological - part-of,
overlap and its dual underlap. Other relations are topological - contact, nontangential part-of, dual
contact and some others definable by means of the contact and part-of relations. This is one of the
reasons that the extension of mereology with these new relations is commonly called
mereotopology. There is no clear difference in literature between RBTS and mereotopology. The
origin of RBTS goes back to Whitehead and de Laguna (\cite{W,deLaguna}). According to Whitehead points, as well as
the other primitive notions in Euclidean geometry such as lines and planes, do not have separate
existence in reality and because of this are not appropriate for primitive notions.
Some papers on RBTS are \cite{V, BD, HG, Pratt-Hartmann2, gerla, V1, ag, gp, tarski, gp1} (also the handbook \cite{A} and \cite{BTV}, containing some logics of
space).

RBTS has applications in computer science because of its simpler way of representing of
qualitative spatial information. Mereotopology is used in the field of Artificial Intelligence, called
Knowledge Representation (KR). RBTS initiated a special field in KR, called Qualitative Spatial
Representation and Reasoning (QSRR) which is appropriate for automatization \cite{B, rn}. RBTS is
applied in geographic information systems, robot navigation. Surveys concerning various
applications are for example \cite{CohnHazarika, CR} and the book \cite{Hazarika} (also special issues of Fundamenta
Informaticae \cite{Du} and the Journal of Applied Nonclassical Logics \cite{Balbiani}). 
One of the most popular systems in Qualitative Spatial Representation and Reasoning is the
Region Connection Calculus (RCC) \cite{RCC}.

The notion of \textit{contact algebra} is one of the main tools in RBTS. This notion appears in the
literature under different names and formulations as an extension of Boolean algebra with some
mereotopological relations \cite{deVries, Stell, vdb, VDDB, BD, dw, dv, DuV}. The simplest system, called just a contact
algebra was introduced in \cite{dv} as an extension of Boolean algebra
$B = (B, 0, 1, \cdot, +, \ast)$ with a binary relation C called \textit{contact} and satisfying five simple axioms:

\noindent
(C1) If $aCb$, then $a \neq 0$,\\
(C2) If $aCb$ and $a\leq c$ and $b\leq d$, then $cCd$,\\
(C3) If $aC(b+c)$, then $aCb$ or $aCc$,\\
(C4) If $aCb$, then $bCa$,\\
(C5) If $a\neq 0$, then $aCa$.

The elements of the Boolean algebra are called regions and are considered as analogs of physical
bodies. Boolean operations are considered as operations for constructing new regions from given
ones. The unit element 1 symbolizes the region containing as its parts all regions, and the zero
element 0 symbolizes the empty region.

The so called \textit{extended contact algebras} (\cite{i,BI}) extend the language of contact algebras by the predicate \textit{covering} which gives the possibility to be defined the predicate \textit{internal connectedness}.

Sometimes there is a problem in the motivation of the operation Boolean complement ($\ast$) of
contact algebra. A question arises - if $a$ represents some region, what region does $a^{\ast}$ represent - it
depends on the universe in which we consider $a$. Moreover if $a$ represents a physical body, then
$a^{\ast}$ is unnatural - such a physical body does not exist. Because of this we can drop the operation of
complement and replace the Boolean part of a contact algebra with distributive lattice. First steps
in this direction were made in \cite{dmvw1, dmvw2}, introducing the notion of distributive contact lattice. In a
distributive contact lattice the only mereotopological relation is the contact relation. Non-tangential inclusion and dual contact (otherwise definable by contact and $\ast$) are not included in the
language. In \cite{iv,i1} the language of distributive contact lattices is extended by considering these two
relations as nondefinable primitives. The well known RCC-8 system of mereotopological relations is definable in this more expressive language and is not definable in the language of distributive contact lattices. 

Furthermore we can generalize contact algebra by dropping also the operation meet. When the elements of a lattice represent physical bodies, the Boolean operation meet ($\cdot$)
gives the closure of the interior of the intersection of two bodies (which in this case coincides with
the intersection of the bodies). In some sense this is an unnatural body and it is reasonable not to consider it. In this paper we eliminate the operation meet from the language of distributive contact
lattices. First we consider \textit{contact join-semilattices (CJS)} and obtain a set-theoretical representation theorem and as a corollary - a topological representation theorem. We define also \textit{distributive contact join-semilattices (DCJS)} and prove that every DCJS is also a CJS. The converse is not true. We obtain also a relational representation theorem for DCJS and as a corollary - a topological one. Finally we define a quantifier-free logic which is decidable.

\section{Preliminaries}

Further we will consider relational and topological contact algebras.

Let $(W,R)$ be a relational system, where $W$ is a nonempty set and $R$ is a reflexive and symmetric binary relation in $W$ and let $B$ be a family of subsets of $W$ closed under union, intersection and complement, containing $\emptyset$ and $W$. We consider the structure $\underline{B}=(B,\leq,0,1,\cdot,+,\ast,C_R)$, where the interpretations of the constants, functional and predicate symbols are the following: $0 = \emptyset$; $1 = W$; $a\leq b$ iff $a \subseteq b$; $a\cdot b = a\cap b$; $a + b = a\cup b$; $a^{\ast} = W\setminus a$; $aC_Rb$ iff $\exists x\in a$ and $\exists y\in b$ such that $xRy$. The obtained structure $\underline{B}$ is called \textbf{relational contact algebra over} $\mathbf{(W,R)}$ \cite{V}.

Topological spaces are among the first mathematical models of space, applied in practice.
Standard models of contact algebras are topological. Let $X$ be a topological space and $a$ be its
subset. We say that $a$ is \textit{regular closed} if $a$ is the closure of its interior. It is a well known fact
that the set $RC(X)$ of all regular closed subsets of $X$ is a Boolean algebra with respect to the
following definitions: $a\leq b$ iff $a\subseteq b$, 0 is the empty set, 1 is the set X, $a+b=a\cup b$, $a\cdot b=Cl\,Int\,(a\cap b)$,
$a^{\ast}=Cl(X \setminus a)$. If we define a contact by $aCb$ iff $a\cap b$ is nonempty, then we obtain a contact
algebra related to $X$, namely $\underline{RC(X)}=(RC(X), \leq , 0, 1, \cdot, +, \ast, C)$ (\cite{dv}, Example 2.1). It is called \textbf{the topological contact algebra over} $\mathbf{X}$.

In the paper we consider also structures which extend by contact relation the language of the \textbf{join-semilattices}, which are defined in the following way:

\begin{definition}\cite{g} 
	\textbf{Join-semilattice} with $0$ is a structure $\underline{L}=(L,\leq,0,+)$ such that are true the axioms
	
	\begin{IEEEeqnarray*}{lll}
		(1)\ x\leq x; & &\\
		(2)\ x\leq y\wedge y\leq x\rightarrow x=y; & &\\
		(3)\ x\leq y\leq z\rightarrow x\leq z; & &\\
		(4)\ x+y=y+x; & &\\
		(5)\ x\leq x+y; & &\\
		(6)\ x,y\leq z\rightarrow x+y\leq z; & &\\
		(7)\ 0\leq x. & &
	\end{IEEEeqnarray*}
	
\end{definition}

\begin{definition}\cite{g} 
	\textbf{Distributive join-semilattice} with $0$ is a join-semilattice with $0$ $\underline{L}=(L,\leq,0,+)$ such that is true the axiom \[(ad)\ x\leq a+b\rightarrow (\exists a'\leq a)(\exists b'\leq b)(x=a'+b').\]
\end{definition}

\begin{definition}\label{def2n}\cite{g} (page 80)
	A nonvoid subset $I$ of a join-semilattice $\underline{L}$ is an \textbf{ideal} iff for $a,b\in L$, we have $a+b\in I$ iff $a$ and $b\in I$.
\end{definition}

\begin{definition}\label{def4n}\cite{g} (page 100)
	A subset $F$ of a join-semilattice $\underline{L}$ is called a \textbf{dual ideal} iff $a\in F$ and $a\leq x$ imply that $x\in F$, and $a,b\in F$ implies that there exists a lower bound $d$ of $\{a,b\}$ such that $d\in F$. 
\end{definition}

\begin{definition}\label{def5n}\cite{g} (page 100)
	An ideal $I$ of a join-semilattice $\underline{L}$ is \textbf{prime} iff $I\not=L$ and $L\setminus I$ is a dual ideal.
\end{definition}

\begin{lemma}\label{lemma1n}\cite{g} (page 100)
	Let $I$ be an ideal and let $F$ be a nonvoid dual ideal of a distributive join-semilattice $\underline{L}$. If $I\cap F=\emptyset$, then exists a prime ideal $P$ of $\underline{L}$ with $I\subseteq P$ and $P\cap F=\emptyset$.
\end{lemma}

\section{Adding contact relation}

We consider additionally the following axioms

\begin{IEEEeqnarray*}{lll}
	(8)\ x\leq 1; & &\\
	(9)\ xCy\rightarrow x\neq 0; & &\\
	(10)\ xCy\rightarrow yCx; & &\\
	(11)\ xC(y+z)\rightarrow xCy\text{ or }xCz; & &\\
	(12)\ xCy,\ y\leq y'\rightarrow xCy'; & &\\
	(13)\ x\neq 0\rightarrow xCx; & &\\
	(14)\text{ for any }m,i\geq 1, & &\\
	A^1_{m,i}:\ xCy,\ x\leq s_1,\ldots,s_m,\ y\leq t_1,\ldots,t_m,\ s_1=s^1_1+\ldots +s^i_1,\ldots,s_m=s^1_m+\ldots +s^i_m, & &\\
	t_1=t^1_1+\ldots +t^i_1,\ldots,t_m=t^1_m+\ldots +t^i_m\rightarrow & &\\ \bigvee_{\substack{l_1=1,\ldots,i \\ \ldots \\ l_m=1,\ldots,i \\ k_1=1,\ldots,i \\ \ldots \\ k_m=1,\ldots,i}}\Big(\bigwedge_{1\leq j\leq u\leq m}s_j^{l_j}Cs_u^{l_u}\wedge\bigwedge_{1\leq j\leq u\leq m}t_j^{k_j}Ct_u^{k_u}\wedge\bigwedge_{\substack{j=1,\ldots,m \\ u=1,\ldots,m}}s_j^{l_j}Ct_u^{k_u}\Big); & &\\
	\\
	(15)\text{ for any }n,i\geq 1, & &\\
	A_{n,i}:\ t\not\leq u,\ t\leq x_1,\ldots,x_n,\ x_1=x^1_1+\ldots +x^i_1,\ldots,x_n=x^1_n+\ldots +x^i_n \rightarrow & &\\
	\bigvee_{\substack{j_1=1,\ldots,i \\ \ldots \\ j_n=1,\ldots,i}}\Big(x^{j_1}_1,\ldots,x^{j_n}_n\not\leq u\wedge\bigwedge_{\substack{k=1,\ldots,n\\l=1,\ldots,n}} x^{j_k}_kCx^{j_l}_l\Big); & &
\end{IEEEeqnarray*}

\begin{definition}
	\textbf{Contact join-semilattice (CJS for short)} is a structure $\underline{B}=(B,\leq,0,1,+,C)$ such that are true the axioms (1),\ldots,(10); (14) and (15).
\end{definition}

\begin{remark}
	
	\begin{itemize}
		
		\item The axiom $A^1_{m,i}$ says that if $a$ is in contact with $b$, $a\leq s_1,\ldots,s_m$, $b\leq t_1,\ldots,t_m$ and $s_1,\ldots,s_m$, $t_1,\ldots,t_m$ are presented as finite joins, then one element can be chosen of every join in such a way that every two chosen elements are in contact;
		
		\item The axiom $A_{n,i}$ says that if $t\not\leq u$, $t\leq a_1,\ldots,a_n$ and $a_1,\ldots,a_n$ are presented as finite joins, then one element can be chosen of every join in such a way that every chosen element is not $\leq u$ and every two chosen elements are in contact.
		
		\item The axiom $A^1_{1,1}$ is $xCy$, $x\leq s_1$, $y\leq t_1$, $s_1=s^1_1$, $t_1=t^1_1\rightarrow s^1_1Ct^1_1$ and obviously $A^1_{1,1}$ is equivalent to the axiom (C2) of contact algebra.
		
		\item The axiom $A^1_{1,2}$ is $xCy$, $x\leq s_1$, $y\leq t_1$, $s_1=s^1_1+s^2_1$, $t_1=t^1_1+t^2_1\rightarrow s^1_1Ct^1_1\vee s^1_1Ct^2_1\vee s^2_1Ct^1_1\vee s^2_1Ct^2_1$. By it we easily obtain that in every CJS is true axiom (11).
		
		\item The axiom $A_{1,1}$ is $t\not\leq u$, $t\leq x_1$, $x_1=x^1_1\rightarrow x^1_1\not\leq u$ and $x^1_1Cx^1_1$. By it, taking $u=0$, we easily obtain that in every CJS is true axiom (13).
		
	\end{itemize}
\end{remark}

\begin{definition}\label{def0}
	\textbf{Distributive contact join-semilattice (DCJS for short)} is a structure $\underline{B}=(B,\leq,0,1,+,C)$ such that are true the axioms (1),\ldots,(13) and the axiom (ad).
\end{definition}

We will prove that every DCJS is also a CJS. Let $\underline{B}$ be a DCJS. We will prove that in $\underline{B}$ are true axioms (14) and (15). For the purpose first we will prove two lemmas.

\begin{lemma}\label{lemma3}
	In $\underline{B}$ is true the formula \[(d_n)\ x\leq a_1+\ldots+a_n\rightarrow(\exists a_1'\leq a_1)\ldots(\exists a_n'\leq a_n)(x=a_1'+\ldots+a_n'),\] where $n\geq 2$.
\end{lemma}
\begin{proof}
	We will prove the lemma by induction on $n$. The base of induction is obvious. Let $n>2$ and $\underline{B}\models d_{n-1}$. We will prove that $\underline{B}\models d_n$. Let $x\leq a_1+\ldots+a_n=a_1+\ldots+(a_{n-1}+a_n)$. By the induction hypothesis, there are $a_1'\leq a_1,\ldots,a_{n-2}'\leq a_{n-2}$, $y\leq a_{n-1}+a_n$ such that $x=a_1'+\ldots+a_{n-2}'+y$. Since $y\leq a_{n-1}+a_n$, by axiom (ad), there are $a_{n-1}'\leq a_{n-1}$, $a_n'\leq a_n$ such that $y=a_{n-1}'+a_n'$. $\Box$
\end{proof}

\begin{lemma}\label{lemma4}
	Let $x=s_1^1+\ldots+s_1^i=\ldots=s_m^1+\ldots+s_m^i$. Then there are $t_1,\ldots,t_n$ such that $x=t_1+\ldots+t_n$ and for every $j\in\{1,\ldots,n\}$, there are $l_1,\ldots,l_m\in\{1,\ldots,i\}$ such that $t_j\leq s_1^{l_1},\ldots,s_m^{l_m}$. 
\end{lemma}
\begin{proof}
	Induction on $m$. The base of induction is trivial. Let $m>1$ and the lemma is true for $m-1$. We will prove that it is true for $m$. Let $x=s_1^1+\ldots+s_1^i=\ldots=s_m^1+\ldots+s_m^i$. By the induction hypothesis, there are $t_1,\ldots,t_n$ such that $x=t_1+\ldots+t_n$ and for every $j\in\{1,\ldots,n\}$, there are $l_1,\ldots,l_{m-1}\in\{1,\ldots,i\}$ such that $t_j\leq s_1^{l_1},\ldots,s_{m-1}^{l_{m-1}}$. Now we consider the finite joins $x=s_m^1+\ldots+s_m^i=t_1+\ldots+t_n$. We have that for every $j\in\{1,\ldots,n\}$, $t_j\leq t_1+\ldots+t_n=s_m^1+\ldots+s_m^i$. Using this fact and Lemma~\ref{lemma3}, we get that for every $j\in\{1,\ldots,n\}$, there are $v_j^1\leq s_m^1,\ldots,v_j^i\leq s_m^i$ such that $t_j=v_j^1+\ldots+v_j^i$. Thus $x=v_1^1+\ldots+v_1^i+\ldots+v_n^1+\ldots+v_n^i$ and for any $j\in\{1,\ldots,n\}$, $k\in\{1,\ldots,i\}$, $v_j^k\leq t_j\leq s_1^{l_1},\ldots,s_{m-1}^{l_{m-1}}$ and $v_j^k\leq s_m^k$. $\Box$
\end{proof}

\begin{lemma}\label{lemma5}
	Let $m$, $i\geq 1$. Then $\underline{B}\models A_{m,i}^1$.
\end{lemma}
\begin{proof}
	Let $xCy$, $x\leq s_1,\ldots,s_m$, $y\leq t_1,\ldots,t_m$, $s_1=s_1^1+\ldots+s_1^i,\ldots,s_m=s_m^1+\ldots+s_m^i$, $t_1=t_1^1+\ldots+t_1^i,\ldots,t_m=t_m^1+\ldots+t_m^i$. Using Lemma~\ref{lemma3}, we obtain that there are $s_{\alpha 2}^{\beta}\leq s_{\alpha}^{\beta}$, $t_{\alpha 2}^{\beta}\leq t_{\alpha}^{\beta}$ for $\alpha=1,\ldots,m$ and $\beta=1,\ldots,i$ such that $x=s_{12}^1+\ldots+s_{12}^i=\ldots=s_{m2}^1+\ldots+s_{m2}^i$; $y=t_{12}^1+\ldots+t_{12}^i=\ldots=t_{m2}^1+\ldots+t_{m2}^i$. Using Lemma~\ref{lemma4}, we obtain that there are $u_1,\ldots,u_n,v_1,\ldots,v_k$ such that $x=u_1+\ldots+u_n$, $y=v_1+\ldots+v_k$; for every $z\in\{1,\ldots,n\}$, there are $l_1,\ldots,l_m\in\{1,\ldots,i\}$ such that $u_z\leq s_{12}^{l_1},\ldots,s_{m2}^{l_m}$; for every $z\in\{1,\ldots,k\}$, there are $j_1,\ldots,j_m\in\{1,\ldots,i\}$ such that $v_z\leq t_{12}^{j_1},\ldots,t_{m2}^{j_m}$. By axiom (11) it can be easily verified that $(u_1+\ldots+u_n)C(v_1+\ldots+v_k)$ implies that there are $z_1\in\{1,\ldots,n\}$, $z_2\in\{1,\ldots,k\}$ such that $u_{z_1}C v_{z_2}$. Clearly there are $l_1,\ldots,l_m\in\{1,\ldots,i\}$ such that $u_{z_1}\leq s_1^{l_1},\ldots,s_m^{l_m}$; there are $j_1,\ldots,j_m\in\{1,\ldots,i\}$ such that $v_{z_2}\leq t_1^{j_1},\ldots,t_m^{j_m}$. Using axiom (12), we get that every element among $s_1^{l_1},\ldots,s_m^{l_m}$ is in contact with every element among $t_1^{j_1},\ldots,t_m^{j_m}$. By $u_{z_1}C v_{z_2}$ and axiom (9), $u_{z_1}\neq 0$ and hence by axiom (13), $u_{z_1}C u_{z_1}$; so using axiom (12) and $u_{z_1}\leq s_1^{l_1},\ldots,s_m^{l_m}$, we obtain that every two elements among $s_1^{l_1},\ldots,s_m^{l_m}$ are in contact. Similarly every two elements among $t_1^{j_1},\ldots,t_m^{j_m}$ are in contact. $\Box$
\end{proof}

\begin{lemma}\label{lemma6}
	Let $n$, $i\geq 1$. Then $\underline{B}\models A_{n,i}$.
\end{lemma}
\begin{proof}
	Let $t\not\leq u$, $t\leq x_1,\ldots,x_n$, $x_1=x_1^1+\ldots+x_1^i,\ldots,x_n=x_n^1+\ldots+x_n^i$. By Lemma~\ref{lemma3} and Lemma~\ref{lemma4}, there are $t_1,\ldots,t_m$ such that $t=t_1+\ldots+t_m$ and for every $j\in\{1,\ldots,m\}$, there are $l_1,\ldots,l_n\in\{1,\ldots,i\}$ such that $t_j\leq x_1^{l_1},\ldots,x_n^{l_n}$. Suppose for the sake of contradiction that $t_1,\ldots,t_m\leq u$. By axiom (6) we get that $t_1+\ldots+t_m\leq u$, i.e. $t\leq u$ - a contradiction. Consequently there is $j\in\{1,\ldots,m\}$ such that $t_j\not\leq u$ and hence $t_j\neq 0$; so $t_jCt_j$. There are $l_1,\ldots,l_n\in\{1,\ldots,i\}$ such that $t_j\leq x_1^{l_1},\ldots,x_n^{l_n}$. Thus every two elements among $x_1^{l_1},\ldots,x_n^{l_n}$ are in contact. Let $k\in\{1,\ldots,n\}$. Suppose for the sake of contradiction that $x_k^{l_k}\leq u$ but we have $t_j\leq x_k^{l_k}$, so $t\leq u$ - a contradiction. Consequently $x_k^{l_k}\not\leq u$. $\Box$
\end{proof}

\bigskip
By Lemma~\ref{lemma5} and Lemma~\ref{lemma6} we obtain

\begin{proposition}\label{pr0nn}
	Every DCJS is also a CJS.
\end{proposition}

\section{Examples of contact join-semilattices and distributive contact join-semilattices}

In this section we will give concrete examples of CJS and DCJS. These examples are considered as "standard examples" because later on we will prove representation theorems of CJS and DCJS by algebras of such standard type.

We will need the following proposition

\begin{proposition}\label{pr1nn}
	Every contact algebra is a DCJS.
\end{proposition}
\begin{proof}
	Let $\underline{B}$ be a contact algebra. Obviously axioms $(1),\ldots,(13)$ are true in $\underline{B}$. We will prove that $\underline{B}\models (ad)$. Let $x\leq a+b$. We have $x\cdot a\leq a$, $x\cdot b\leq b$ and $x=x\cdot(a+b)=x\cdot a+x\cdot b$, because $\underline{B}$ is a distributive lattice. $\Box$
\end{proof}

\bigskip
The following lemma shows a set-theoretical example of CJS

\begin{lemma}\label{lemma5nn}
	Let $W$ be a nonempty set and $B$ be a family of subsets of $W$, containing $\emptyset$, $W$ and closed under $\cup$. We define in $B$: $0=\emptyset$, $1=W$, $a+b=a\cup b$, $a\leq b$ iff $a\subseteq b$, $aCb$ iff $a\cap b\not=\emptyset$. Then the obtained structure $\underline{B}=(B,\leq,0,1,+,C)$ is a CJS.
\end{lemma}
\begin{proof}
	We consider $\underline{B_1}$ - the relational contact algebra over $(W,=)$, where $B_1=2^W$. From Proposition~\ref{pr1nn} we get that $\underline{B_1}$ is a DCJS and by Proposition~\ref{pr0nn}, $\underline{B_1}$ is a CJS. Clearly $\underline{B}$ is a substructure of $\underline{B_1}$. But we also have that the axioms of CJS can be considered as universal formulas and therefore $\underline{B}$ is also a CJS. $\Box$
\end{proof}

\bigskip
The following lemma shows a relational example of DCJS

\begin{lemma}\label{lemma6nn}
	Let $(W,R)$ be a relational system with a reflexive and symmetric relation $R$ and let $B$ be a family of subsets of $W$, containing $\emptyset$, $W$ and closed under $\cup$. We define in $B$: $0=\emptyset$, $1=W$, $a+b=a\cup b$, $a\leq b$ iff $a\subseteq b$, $aCb$ iff $(\exists U\in a)(\exists V\in b)(URV)$. If in $\underline{B}=(B,\leq,0,1,+,C)$ is fulfilled the axiom (ad), then $\underline{B}$ is a DCJS.
\end{lemma}
\begin{proof}
	We consider $\underline{B_1}$ - the relational contact algebra over $(W,R)$, where $B_1=2^W$. By Proposition~\ref{pr1nn}, $\underline{B_1}$ is a DCJS. Clearly $\underline{B}$ is a substructure of $\underline{B_1}$. Axioms $(1),\ldots,(13)$ can be considered as universal formulas, they are true in $\underline{B_1}$ (since $\underline{B_1}$ is a DCJS); so they are true also in the substructure $\underline{B}$. We have that in $\underline{B}$ is true (ad) and consequently $\underline{B}$ is a DCJS. $\Box$
\end{proof}

\bigskip
The following lemmas show topological examples of CJS and of DCJS

\begin{lemma}\label{lemma7nn}
	Let $X$ be a topological space and $B$ be a subset of $RC(X)$, containing $\emptyset$, $X$ and closed under $\cup$. We define in $B$: $0=\emptyset$, $1=X$, $a+b=a\cup b$, $a\leq b$ iff $a\subseteq b$, $aCb$ iff $a\cap b\neq\emptyset$. Then the obtained structure $\underline{B}=(B,\leq,0,1,+,C)$ is a CJS.
\end{lemma}
\begin{proof}
	Clearly $\underline{B}$ is a substructure of the topological contact algebra over $X$ and similarly as in the proof of Lemma~\ref{lemma5nn} we get that $\underline{B}$ is a CJS. $\Box$
\end{proof}

\begin{lemma}\label{lemma8nn}
		Let $X$ be a topological space and $B$ be a subset of $RC(X)$, containing $\emptyset$, $X$ and closed under $\cup$. We define in $B$: $0=\emptyset$, $1=X$, $a+b=a\cup b$, $a\leq b$ iff $a\subseteq b$, $aCb$ iff $a\cap b\neq\emptyset$. If $\underline{B}=(B,\leq,0,1,+,C)$ satisfies the axiom (ad), then $\underline{B}$ is a DCJS.
\end{lemma}
\begin{proof}
	The proof is similar to the proof of Lemma~\ref{lemma7nn}, using that $\underline{B}$ is a substructure of $\underline{RC(X)}$. $\Box$
\end{proof}

\begin{proposition}\label{pr2nn}
	There is a standard set-theoretical example of CJS which is not a DCJS.
\end{proposition}
\begin{proof}
	We consider the set $W=\{1,2,3,4\}$. Let $B=\{\emptyset, W, \{1,3\}, \{2,4\}, \{1,2\}, \{1,2,3\}, \{1,2,4\}\}$. It can be easily verified that $B$ is closed under $\cup$. We define in $B$: $0=\emptyset$, $1=W$, $a+b=a\cup b$, $a\leq b$ iff $a\subseteq b$, $aCb$ iff $a\cap b\neq\emptyset$. By Lemma~\ref{lemma5nn}, the structure $\underline{B}=(B,\leq,0,1,+,C)$ is a CJS. But $\underline{B}$ does not satisfy the axiom (ad), because $\{1,2\}\leq\{1,3\}+\{2,4\}$ but $(\forall a'\leq\{1,3\})(\forall b'\leq\{2,4\})(\{1,2\}\neq a'+b')$. $\Box$
\end{proof}

\begin{proposition}\label{pr3nn}
	There is a standard topological example of CJS which is not a DCJS.
\end{proposition}
\begin{proof}
	We consider the same $W$ and $B$ as in the proof of Proposition~\ref{pr2nn}. We define topology on $W$, taking for open all subsets of $W$. It can be easily verified that $RC(W)=2^W$. We define in $B$: $0=\emptyset$, $1=W$, $a+b=a\cup b$, $a\leq b$ iff $a\subseteq b$, $aCb$ iff $a\cap b\neq\emptyset$. By Lemma~\ref{lemma7nn}, the obtained structure $\underline{B}=(B,\leq,0,1,+,C)$ is a CJS. The structure $\underline{B}$ is the same as the structure $\underline{B}$ in the proof of Proposition~\ref{pr2nn} and therefore $\underline{B}$ is not a DCJS. $\Box$
\end{proof}

\section{Representation theorems for contact join-semilattices}

First we will prove a set-theoretical representation theorem of CJS. For this purpose we will need the following definition, taken from the theory of contact algebras

\begin{definition}\label{def2}\cite{dv}
	Let $\underline{B}$ be a CJS. A subset of $B$ $\Gamma$ is called a \textbf{clan} in
	$\underline{B}$ if the following conditions are true:\\
	1) $1\in\Gamma$;\\
	2) $0\notin\Gamma$;\\
	3) $x\in\Gamma$, $x\leq y\rightarrow y\in\Gamma$;\\
	4) $x,y\in\Gamma\rightarrow x C y$;\\
	5) $x+y\in\Gamma\rightarrow x\in\Gamma$ or $y\in\Gamma$.\\
	We denote by $Clans(\underline{B})$ the set of the clans in $\underline{B}$.
\end{definition}

\begin{exa}
	Let $W$ be a nonempty set and $\underline{B}$ be the standard set-theoretical example of CJS of all subsets of $W$. Let $x\in W$. Then it can be easily verified that $P_x=\{P\subseteq W:\ x\in P\}$ is a clan.
\end{exa}

Let $\underline{B}$ be an arbitrary CJS. We will prove several lemmas. The first lemma has two variants - the first variant contains the text in the brackets, the second one - no.

\begin{lemma}\label{lemma1}
	(Let $u\neq 1$.) Let $\Gamma$ be a subset of $B$ and $\Gamma$ satisfies condition 3) from Definition~\ref{def2} and the condition:\\
	$(\ast)\ x_1,\ldots,x_n\in\Gamma\rightarrow$ for every presentation of $x_1,\ldots,x_n$ as finite joins, one element can be chosen of every join $(\not\leq u)$ in such a way that every two chosen elements are in contact.
	
	Let $x+y\in\Gamma$. Then there exists a set $\Gamma_1$, satisfying the same conditions and such that $\Gamma_1=\Gamma\cup\{z:\ x\leq z\}$ or $\Gamma_1=\Gamma\cup\{z:\ y\leq z\}$.
\end{lemma}
\begin{proof}
	We will prove only the first variant of the lemma. The second variant is proved similarly. Suppose for the sake of contradiction that the following two conditions are true:\\
	$(\clubsuit)$ there are $x_1,\ldots,x_m\in\Gamma$, $z_1,\ldots,z_k\geq x$ and presentations of $x_1,\ldots,x_m,z_1,\\\ldots,z_k$ as finite joins such that it is impossible to be chosen one element $\not\leq u$ of every join in such a way that every two chosen are in contact;\\
	$(\spadesuit)$ there are $y_1,\ldots,y_n\in\Gamma$, $t_1,\ldots,t_r\geq y$ and presentations of $y_1,\ldots,y_n,t_1,\ldots,\\t_r$ as finite joins such that it is impossible to be chosen one element $\not\leq u$ of every join in such a way that every two chosen are in contact.
	
	Let the presentations as finite joins be:
	\begin{IEEEeqnarray*}{lll}
	x_1=x^1_1+\ldots+x^{i_1}_1    & \qquad\qquad\qquad & y_1=y^1_1+\ldots+y^{j_1}_1\\
	\vdots                        & \qquad\qquad\qquad & \vdots\\
	x_m=x^1_m+\ldots+x^{i_m}_m    & \qquad\qquad\qquad & y_n=y^1_n+\ldots+y^{j_n}_n\\
	z_1=z^1_1+\ldots+z^{i_{11}}_1 & \qquad\qquad\qquad & t_1=t^1_1+\ldots+t^{j_{11}}_1\\
	\vdots                        & \qquad\qquad\qquad & \vdots\\
	z_k=z^1_k+\ldots+z^{i_{1k}}_k & \qquad\qquad\qquad & t_r=t^1_r+\ldots+t^{j_{1r}}_r
	\end{IEEEeqnarray*}
	Let $i\in\{1,\ldots,k\}$, $j\in\{1,\ldots,r\}$. It can be easily verified that $x+y\leq z_i+t_j$. We have also that $x+y\in\Gamma$. Consequently $z_i+t_j\in\Gamma$. We have that $x_1,\ldots,x_m,y_1,\ldots,y_n,z_1+t_1,\ldots,z_1+t_r,\ldots,z_k+t_1,\ldots,z_k+t_r\in\Gamma$. Thus by property $(\ast)$ of $\Gamma$ we obtain that one element can be chosen of every of the following joins $(\not\leq u)$:
	\begin{IEEEeqnarray*}{ll}
		x^1_1+\ldots+x^{i_1}_1 & \qquad\qquad\qquad\qquad\qquad\qquad\qquad\qquad\\
		\vdots & \qquad\qquad\qquad\qquad\qquad\qquad\qquad\qquad\\
		x^1_m+\ldots+x^{i_m}_m & \qquad\qquad\qquad\qquad\qquad\qquad\qquad\qquad\\
		y^1_1+\ldots+y^{j_1}_1 & \qquad\qquad\qquad\qquad\qquad\qquad\qquad\qquad\\
		\vdots & \qquad\qquad\qquad\qquad\qquad\qquad\qquad\qquad\\
		y^1_n+\ldots+y^{j_n}_n & \qquad\qquad\qquad\qquad\qquad\qquad\qquad\qquad\\ 
		z^1_1+\ldots+z^{i_{11}}_1+t^1_1+\ldots+t^{j_{11}}_1 & \qquad\qquad\qquad\qquad\qquad\qquad\qquad\qquad\\
		\vdots & \qquad\qquad\qquad\qquad\qquad\qquad\qquad\qquad\\
		z^1_1+\ldots+z^{i_{11}}_1+t^1_r+\ldots+t^{j_{1r}}_r & \qquad\qquad\qquad\qquad\qquad\qquad\qquad\qquad\\ 
		\vdots & \qquad\qquad\qquad\qquad\qquad\qquad\qquad\qquad\\
		z^1_k+\ldots+z^{i_{1k}}_k+t^1_1+\ldots+t^{j_{11}}_1 & \qquad\qquad\qquad\qquad\qquad\qquad\qquad\qquad\\
		\vdots & \qquad\qquad\qquad\qquad\qquad\qquad\qquad\qquad\\
		z^1_k+\ldots+z^{i_{1k}}_k+t^1_r+\ldots+t^{j_{1r}}_r & \qquad\qquad\qquad\qquad\qquad\qquad\qquad\qquad
	\end{IEEEeqnarray*}
	in such a way that every two chosen are in contact. Suppose for the sake of contradiction that for every $s\in\{1,\ldots,k\}$, among $z^1_s,\ldots,z^{i_{1s}}_s$ some element is chosen. But this is a contradiction with $(\clubsuit)$. Consequently there is $s\in\{1,\ldots,k\}$ such that among $z^1_s,\ldots,z^{i_{1s}}_s$ no one is chosen. Consequently one element $(\not\leq u)$ is chosen from every of the joins:
	\begin{IEEEeqnarray*}{ll}
		t^1_1+\ldots+t^{j_{11}}_1 & \qquad\qquad\qquad\qquad\qquad\qquad\qquad\qquad\\
		\vdots & \qquad\qquad\qquad\qquad\qquad\qquad\qquad\qquad\\
		t^1_r+\ldots+t^{j_{1r}}_r & \qquad\qquad\qquad\qquad\qquad\qquad\qquad\qquad\\
		y^1_1+\ldots+y^{j_1}_1 & \qquad\qquad\qquad\qquad\qquad\qquad\qquad\qquad\\
		\vdots & \qquad\qquad\qquad\qquad\qquad\qquad\qquad\qquad\\
		y^1_n+\ldots+y^{j_n}_n & \qquad\qquad\qquad\qquad\qquad\qquad\qquad\qquad
	\end{IEEEeqnarray*}
	in such a way that every two chosen elements are in contact. But this is a contradiction with $(\spadesuit)$. Consequently condition $(\clubsuit)$ is not true or condition $(\spadesuit)$ is not true. Without loss of generality $(\clubsuit)$ is not true i.e. the following condition is satisfied:\\
	$(\heartsuit)$ for any $x_1,\ldots,x_m\in\Gamma$, $z_1,\ldots,z_k\geq x$ and presentations of $x_1,\ldots,x_m$, $z_1,\ldots,z_k$ as finite joins, one element $\not\leq u$ can be chosen of every join in such a way that every two chosen elements are in contact.
	
	We consider the set $\Gamma_1=\Gamma\cup\{z:\ x\leq z\}$. It can be easily proved that $\Gamma_1$ satisfies property 3) from Definition~\ref{def2}. We will prove that $\Gamma_1$ satisfies the property $(\ast)$. Let $a_1,\ldots,a_p,b_1,\ldots,b_q\in\Gamma_1$, where $p+q>0$, $a_1,\ldots,a_p\in\Gamma$, $b_1,\ldots,b_q\geq x$. We will prove that for every presentation of $a_1,\ldots,a_p,b_1,\ldots,b_q$ as finite joins, one element ($\not\leq u$) can be chosen of every join in such a way that every two chosen are in contact.\\
	\textbf{Case 1:} $q=0$\\
	The proof is obvious.\\
	\textbf{Case 2:} $p=0$\\
	Let us have the following presentation of $b_1,\ldots,b_q$ as finite joins:
	\begin{IEEEeqnarray*}{ll}
		b_1=b^1_1+\ldots+b^{l_1}_1 & \qquad\qquad\qquad\qquad\qquad\qquad\qquad\qquad\\
		\vdots & \qquad\qquad\qquad\qquad\qquad\qquad\qquad\qquad\\
		b_q=b^1_q+\ldots+b^{l_q}_q & \qquad\qquad\qquad\qquad\qquad\qquad\qquad\qquad
	\end{IEEEeqnarray*}
	We have also $x+y\in\Gamma$ and we finish the proof, using condition $(\heartsuit)$.\\
	\textbf{Case 3:} $p,q>0$\\
	Again we use condition $(\heartsuit)$. $\Box$
\end{proof}

\begin{lemma}\label{lemma20}
	Let $tCt_1$. Then there is a clan $\Gamma$ such that $t$, $t_1\in\Gamma$.
\end{lemma}
\begin{proof}
	We consider $M=\{P\subseteq B:\\
	t,\ t_1\in P;\\
	0\not\in P;\\
	x\in P,\ x\leq y\rightarrow y\in P;\\
	x_1,\ldots,x_k\in P\rightarrow\text{for every presentation of }x_1,\ldots,x_k\text{ as finite joins, one element}\\
	\text{can be chosen of every join in such a way that every two chosen elements are}\\ 
	\text{in contact}\}$.\\
    We will prove that $(M,\subseteq)$ has a maximal element. Let $L$ be a chain in $(M,\subseteq)$. We will prove that $L$ has an upper bound in $M$.\\
    \textbf{Case 1:} $L=\emptyset$\\
    We consider the set $P=\{x\in B:\ t\leq x\text{ or }t_1\leq x\}$. We will prove that $P\in M$. For the purpose we will prove only the last condition for the elements of $M$. The other conditions are obviously true. Let $x_1,\ldots,x_k\in P$. Let $\{x\in \{x_1,\ldots,x_k\}:\ t\leq x\}=\{a_1,\ldots,a_m\}$, where $m\geq 0$. Let $b_1,\ldots,b_n\ (n\geq 0)$ be the rest elements of $\{x_1,\ldots,x_k\}$, i.e. $t_1\leq b_1,\ldots,b_n$. 
    
    Let us have the following presentations of $a_1,\ldots,a_m,b_1,\ldots,b_n$ as finite joins:
    \begin{IEEEeqnarray*}{ll}
    	a_1=a^1_1+\ldots+a^{i_1}_1 & \qquad\qquad\qquad\qquad\qquad\qquad\qquad\qquad\\
    	\vdots & \qquad\qquad\qquad\qquad\qquad\qquad\qquad\qquad\\
    	a_m=a^1_m+\ldots+a^{i_m}_m & \qquad\qquad\qquad\qquad\qquad\qquad\qquad\qquad\\
    	b_1=b^1_1+\ldots+b^{j_1}_1 & \qquad\qquad\qquad\qquad\qquad\qquad\qquad\qquad\\
    	\vdots & \qquad\qquad\qquad\qquad\qquad\qquad\qquad\qquad\\
    	b_n=b^1_n+\ldots+b^{j_n}_n
    \end{IEEEeqnarray*}
    
    We will consider only the case when $m\geq n$. The other case $(n\geq m)$ is symmetric. We have \[tCt_1,\ t\leq a_1,\ldots,a_m,\ t_1\leq\underbrace{b_1,\ldots,b_n,1,\ldots,1}_{m\text{ times}},\ m>0\] Let $i=max(i_1,\ldots,i_m,j_1,\ldots,j_n)$. We supplement every join with its first element in such a way that to have $i$ elements. We use also that $1=\underbrace{1+\ldots+1}_{i\text{ times}}$. By axiom $A^1_{m,i}$ we get that one element can be chosen from the new joins in such a way that every two chosen elements are in contact. Consequently one element can be chosen from every of the initial joins in such a way that every two chosen elements are in contact. Consequently $P\in M$. $P$ is an upper bound of $L$.\\
    \textbf{Case 2:} $L\neq\emptyset$\\
    It can be easily verified that $\bigcup L\in M$. Obviously $\bigcup L$ is an upper bound of $L$.
    
    By Zorn Lemma, $(M,\subseteq)$ has a maximal element $\Gamma$. We will prove that $\Gamma$ is a clan. It is easily seen that $\Gamma$ satisfies conditions $1),\ldots,4)$ of Definition~\ref{def2}. Now we will prove that $\Gamma$ satisfies condition 5) of Definition~\ref{def2}. Let $x+y\in\Gamma$. By the second variant of Lemma~\ref{lemma1}, without loss of generality there exists a set $\Gamma_1$ such that satisfies properties 3) and 4) of Definition~\ref{def2}, the last condition of the definition of $M$ and $\Gamma_1=\Gamma\cup\{z:\ x\leq z\}$. We will prove that $\Gamma_1\in M$. Since $\Gamma\in M$, $t$, $t_1\in\Gamma$ and hence $t$, $t_1\in\Gamma_1$. Suppose for the sake of contradiction that $0\in\Gamma_1$. Since $\Gamma_1$ satisfies condition 4) of Definition~\ref{def2}, $0C0$ and hence $0\neq 0$ - a contradiction. Consequently $0\notin\Gamma_1$. Clearly $\Gamma_1$ satisfies the rest conditions of the definition of $M$. Consequently $\Gamma_1\in M$. We have also that $\Gamma$ is a maximal element of $M$ and $\Gamma\subseteq\Gamma_1$. Thus $\Gamma=\Gamma_1$. Clearly $x\in\Gamma_1$. Consequently $x\in\Gamma$. Thus $\Gamma$ satisfies condition 5) of Definition~\ref{def2}; so $\Gamma$ is a clan. We have that $\Gamma\in M$ and therefore $t$, $t_1\in\Gamma$. $\Box$
    \end{proof}
    
    \begin{lemma}\label{lemma2}
    	Let $t\not\leq u$. Then there is a clan $\Gamma$ such that $t\in\Gamma$, $u\notin\Gamma$. 
    \end{lemma}
    \begin{proof}
    	We consider the set $M=\{P\subseteq B:\\
    	t\in P,\ u\notin P;\\
    	x\in P,\ x\leq y\rightarrow y\in P;\\
    	x_1,\ldots,x_k\in P\rightarrow\text{ for every presentation of }x_1,\ldots,x_k\text{ as finite joins, one element}\\
    	\text{can be chosen of every join, }\not\leq u\text{, in such a way that every two chosen elements}\\
    	\text{are in contact}\}$.\\
    	We will prove that $(M,\subseteq)$ has a maximal element. Let $L$ be a chain in $(M,\subseteq)$. We will prove that $L$ has an upper bound in $M$.\\
    	\textbf{Case 1:} $L=\emptyset$\\
    	We consider the set $P=\{x\in B:\ t\leq x\}$. We will prove that $P\in M$. The conditions for the elements of $M$, without the last one, are obviously true. We will prove the last condition. Let $x_1,\ldots,x_k\in P$. Let us have the following presentations of $x_1,\ldots,x_k$ as finite joins:
    	\begin{IEEEeqnarray*}{ll}
    		x_1=x^1_1+\ldots+x^{i_1}_1 & \qquad\qquad\qquad\qquad\qquad\qquad\qquad\qquad\\
    		\vdots & \qquad\qquad\qquad\qquad\qquad\qquad\qquad\qquad\\
    		x_k=x^1_k+\ldots+x^{i_k}_k & \qquad\qquad\qquad\qquad\qquad\qquad\qquad\qquad
    	\end{IEEEeqnarray*}
    	Let $i=max(i_1,\ldots,i_k)$. We supplement every join with its first element in such a way that all joins to have $i$ elements. By axiom $A_{k,i}$, one element can be chosen of every join, $\not\leq u$, in such a way that every two chosen elements are in contact. Consequently the last condition of the definition of $M$ is fulfilled. Thus $P\in M$. The set $P$ is an upper bound of $L$.\\
    	\textbf{Case 2:} $L\neq\emptyset$\\
    	It can be easily verified that $\bigcup L\in M$. Clearly $\bigcup L$ is an upper bound of $L$.
    	
    	From Zorn Lemma we obtain  that $(M,\subseteq)$ has a maximal element $\Gamma$. We will prove that $\Gamma$ is a clan. It can be easily verified that $\Gamma$ fulfills conditions $1),\ldots,4)$ of Definition~\ref{def2}. We will prove that $\Gamma$ satisfies condition $5)$ of Definition~\ref{def2}. Let $x+y\in\Gamma$. We must prove that $x\in\Gamma$ or $y\in\Gamma$. Since $t\not\leq u$, $u\neq 1$. By the first variant of Lemma~\ref{lemma1}, without loss of generality there exists a set $\Gamma_1$ such that fulfills property 3) of Definition~\ref{def2}, the property $(\ast)$ from the first variant of Lemma~\ref{lemma1} and $\Gamma_1=\Gamma\cup\{z:\ x\leq z\}$. We will prove that $\Gamma_1\in M$. Suppose for the sake of contradiction that $u\in\Gamma_1$. Since $\Gamma_1$ fulfills property $(\ast)$, $u\in\Gamma_1$, $u=u$ (a presentation of $u$ as a finite join), we have $u\not\leq u$ - a contradiction. Consequently $u\notin\Gamma_1$. The rest conditions of the definition of $M$ can be verified easily. Consequently $\Gamma_1\in M$. Thus $\Gamma_1=\Gamma$ and $x\in\Gamma$. Consequently $\Gamma$ satisfies condition 5) of Definition~\ref{def2}. Thus $\Gamma$ is a clan. We have $\Gamma\in M$ and hence $t\in\Gamma$ and $u\notin\Gamma$. $\Box$
    \end{proof}
    
    \bigskip
    Now we can prove
    
    \begin{theorem}[Set-theoretical representation theorem of CJS]\label{th1CJS} 
    	Let $\underline{B}$ be a CJS. Then there is a nonempty set $W$ and an isomorphic embedding of $\underline{B}$ in the standard set-theoretical example of CJS of all subsets of $W$.
    \end{theorem}
    \begin{proof}
    	Let $W=Clans(\underline{B})$. We define a function $h$ from $B$ to $2^W$ in the following way: $h(a)=\{\Gamma\in Clans(\underline{B}):\ a\in\Gamma\}$. We will prove that $h$ is an isomorphic embedding.
    	
    	We will show that $h$ is an injection. Let $a\neq b$. Suppose for the sake of contradiction that $a\leq b$ and $b\leq a$. By axiom (2), $a=b$ - a contradiction. Consequently $a\not\leq b$ or $b\not\leq a$. Without loss of generality $a\not\leq b$. By Lemma~\ref{lemma2}, there is a clan $\Gamma$ such that $a\in\Gamma$, $b\notin\Gamma$. Consequently $\Gamma\in h(a)$ and $\Gamma\notin h(b)$, i.e. $h(a)\neq h(b)$.
    	 
    	Clearly $h(0)=\emptyset$ and $h(1)=W$. 
    	
    	We will prove that $h$ preserves the operation $+$. By condition 5) from Definition~\ref{def2}, $h(a+b)\subseteq h(a)\cup h(b)$. Let $\Gamma\in h(a)\cup h(b)$. We will prove $\Gamma\in h(a+b)$. Without loss of generality $\Gamma\in h(a)$ and hence $\Gamma\in Clans(\underline{B})$, $a\in\Gamma$; so using condition 3) from Definition~\ref{def2} and $a\leq a+b$, we obtain that $a+b\in\Gamma$ and therefore $\Gamma\in h(a+b)$. Consequently $h(a+b)=h(a)\cup h(b)$. 
    	
    	We will prove that $h$ preserves the relation $\leq$. We have $a\leq b$ iff $a+b=b$, $h(a)\subseteq h(b)$ iff $h(a)\cup h(b)=h(b)$, $h$ preserves the operation $+$ and $h$ is an injection, so $h$ preserves the relation $\leq$.

    	We will prove that $h$ preserves the relation $C$. We have $h(a)C h(b)\Longleftrightarrow h(a)\cap h(b)\neq\emptyset$. We must prove $aCb$ iff $h(a)\cap h(b)\neq\emptyset$. By Lemma~\ref{lemma20}, $aCb$ implies $h(a)\cap h(b)\neq\emptyset$. Let $h(a)\cap h(b)\neq\emptyset$. Consequently there is $\Gamma\in h(a)$, $h(b)$ and therefore $\Gamma$ is a clan, $a\in\Gamma$ and $b\in\Gamma$. By condition 4) from Definition~\ref{def2}, $aCb$. 
    	
    	Thus $h$ is an isomorphic embedding. $\Box$
    	\end{proof}
    	
    	\begin{theorem}[Topological representation theorem of CJS]\label{th2CJS}
    		Let $\underline{B}$ be a CJS. Then there is a compact, semiregular, $T_0$ topological space $X$ and an isomorphic embedding $h$ of $\underline{B}$ in the topological contact algebra over $X$ (considered as a standard topological example of CJS). 
    	\end{theorem}
    	\begin{proof}
    		As in the proof of Theorem~\ref{th1CJS}, we see that there is a relational system $(W,=)$ and an isomorphic embedding $h_1$ of $\underline{B}$ in the relational contact algebra $\underline{B_1}$ of all subsets of $W$. It is shown in \cite{dv} (Theorem 5.1) that every contact algebra is isomorphically embedded in the topological contact algebra over some compact, semiregular, $T_0$ topological space. Therefore there is an embedding $h_2$ of $\underline{B_1}$ in the topological contact algebra over some compact, semiregular, $T_0$ topological space $X$. The desired embedding $h$ is $h_2\circ h_1$. $\Box$ 
    	\end{proof}

    	\section{Representation theorems of distributive contact join-semilattices}
    	
    	For proving a relational representation theorem of DCJS we will need the following definition
    	
    	\begin{definition}\label{def0n}
    		Let $\underline{B}$ be a DCJS. We define \textbf{abstract point} of $\underline{B}$ as a subset of $B$ $\Gamma$ such that:\\
    		1) $1\in\Gamma$;\\
    		2) $0\notin\Gamma$;\\
    		3) $x\in\Gamma$, $x\leq y\rightarrow y\in\Gamma$;\\
    		4) $x$, $y\in\Gamma\rightarrow$ there is a lower bound of $\{x,y\}$ $z$ such that $z\in\Gamma$;\\
    		5) $x+y\in\Gamma\rightarrow x\in\Gamma$ or $y\in\Gamma$.
    		
    		We denote by $AP(\underline{B})$ the set of all abstract points of $\underline{B}$.
    	\end{definition}
    	
    	We consider an arbitrary DCJS $\underline{B}$. We will prove several lemmas.
    	
    	\begin{lemma}\label{lemmastar}
    		Let $P$ be a prime ideal, $0\in P$, $1\notin P$. Then $U=B\setminus P$ is an abstract point.
    	\end{lemma}
    	\begin{proof}
    		By Definition~\ref{def5n}, $U$ is a dual ideal. Consequently $U$ satisfies conditions 3) and 4) of Definition~\ref{def0n}. Obviously $U$ fulfills conditions 1) and 2) of Definition~\ref{def0n}. Let $x+y\in U$. Suppose for the sake of contradiction that $x$, $y\notin U$. Consequently $x$, $y\in P$ but $P$ is a prime ideal, so $P$ is an ideal, so $x+y\in P$ - a contradiction. Consequently $x\in U$ or $y\in U$. Consequently $U$ satisfies condition 5) of Definition~\ref{def0n}. Thus $U$ is an abstract point. $\Box$ 
    	\end{proof}

    	\begin{lemma}\label{lemma2n}
    		Let $\Gamma$ be a clan and $a\in\Gamma$. Then there is an abstract point $U$ such that $a\in U$, $U\subseteq\Gamma$.
    	\end{lemma}
    	\begin{proof}
    		We consider the set $[a)\stackrel{def}{=}\{x:\ a\leq x\}$. It can be easily verified that $[a)$ is a dual ideal and $[a)\subseteq\Gamma$. We denote $I=B\setminus\Gamma$. Since $[a)\subseteq\Gamma$, $[a)\cap I=\emptyset$. It can be easily verified that $I$ is an ideal. By Lemma~\ref{lemma1n}, there exists a prime ideal $P$ of $\underline{B}$ with $I\subseteq P$ and $P\cap [a)=\emptyset$. We denote $U=B\setminus P$. We have $P\cap [a)=\emptyset$, so $[a)\subseteq U$, so $1\in U$, so $1\notin P$. Since $\Gamma$ is a clan, $0\in I$, so $0\in P$. By Lemma~\ref{lemmastar}, $U$ is an abstract point. Clearly $a\in U$ and $U\subseteq\Gamma$. $\Box$
    	\end{proof}
    	
    	\begin{lemma}\label{lemma3n}
    		Let $\Gamma$ be a clan. Then there is a set of abstract points $\Sigma$ such that $\Gamma=\bigcup\Sigma$ and for any $U$, $V\in\Sigma$, $x\in U$ and $y\in V$ imply $xCy$. 
    	\end{lemma}
    	\begin{proof}
    		Let $a\in\Gamma$. By Lemma~\ref{lemma2n}, there is an abstract point $U_a$ such that $a\in U_a$, $U_a\subseteq\Gamma$. We denote $\Sigma=\{U_a:\ a\in\Gamma\}$. It can be easily verified that $\Gamma=\bigcup\Sigma$. Let $U$, $V\in\Sigma$. Let $x\in U$, $y\in V$. We must prove that $xCy$. Since $U$, $V\in\Sigma$, $U=U_b$ and $V=U_c$ for some $b$, $c\in\Gamma$ and moreover $U_b$, $U_c\subseteq\Gamma$. Consequently $x$, $y\in\Gamma$ but $\Gamma$ is a clan, so $xCy$. $\Box$
    	\end{proof}
    	
    	\begin{lemma}\label{lemma30n0}
    		Every two elements of an abstract point are in contact.
    	\end{lemma}
    	\begin{proof}
    		The lemma can be easily proved using axioms (13), (12) and (10) from Definition~\ref{def0}. $\Box$
    	\end{proof}
    	
    	\begin{corollary}\label{lemma30n}
    		Every abstract point is a clan.
    	\end{corollary}
    	
    	\begin{lemma}\label{lemma4n}
    		Let $t\not\leq u$. Then there is an abstract point $U$ such that $t\in U$, $u\notin U$. 
    	\end{lemma}
    	\begin{proof}
    		Since $\underline{B}$ is a DCJS, $\underline{B}$ is a CJS and we can apply Lemma~\ref{lemma2}. Thus there is a clan $\Gamma$ such that $t\in\Gamma$, $u\notin\Gamma$. By Lemma~\ref{lemma2n}, there is an abstract point $U$ such that $t\in U$, $U\subseteq\Gamma$. Obviously $u\notin U$. $\Box$
    	\end{proof}
    	
    	\bigskip
    	Now we can prove
    	
    	\begin{theorem}[Relational representation theorem of DCJS]\label{th1n}
    		Let $\underline{B}$ be a DCJS. Then there is a relational structure $(W,R)$ with a reflexive and symmetric relation $R$ and an isomorphic embedding of $\underline{B}$ in the relational contact algebra of all subsets of $W$ (considered as the standard relational example of DCJS of all subsets of $W$). 	
    	\end{theorem}
    	\begin{proof}
    		Let $W=AP(\underline{B})$. We define $R$ in the following way: $URV$ iff $(\forall a\in U)(\forall b\in V)(aCb)$. 
    		
    		By Corollary~\ref{lemma30n}, $R$ is reflexive. Obviously $R$ is symmetric. We define a function $h$ from $B$ to $2^W$ in the following way: $h(a)=\{U\in AP(\underline{B}):\ a\in U\}$. We will prove that $h$ is an isomorphic embedding.
    		
    		Using Lemma~\ref{lemma4n}, we prove that $h$ is an injection.
    		
    		Clearly $h(0)=\emptyset$ and $h(1)=W$.
    		
    		Similarly as in Theorem~\ref{th1CJS} we prove that $h$ preserves the operation $+$ and the relation $\leq$. 
    		
    		We will prove that $h$ preserves the relation $C$. Let $a$, $b\in B$. We have $h(a)Ch(b)$ iff there are $U\in h(a)$, $V\in h(b)$ such that $(\forall x\in U)(\forall y\in V)(xCy)$. Clearly $h(a)Ch(b)$ implies $aCb$. Now let $aCb$. Since $\underline{B}$ is also a CJS, using Lemma~\ref{lemma20}, we obtain that there is a clan $\Gamma$ such that $a$, $b\in\Gamma$. By Lemma~\ref{lemma3n} we see that $h(a)Ch(b)$. Consequently $h$ preserves the relation $C$.
    		
    		Thus $h$ is an isomorphic embedding. $\Box$
    	\end{proof}
    	
    	\begin{theorem}[Topological representation theorem of DCJS]\label{th2n}
    		Let $\underline{B}$ be a DCJS. Then there is a compact, semiregular, $T_0$ topological space $X$ and an isomorphic embedding $h$ of $\underline{B}$ in the topological contact algebra over $X$ (considered as a standard topological example of DCJS).
    	\end{theorem}
    	\begin{proof}
    		The proof is similar of the proof of Theorem~\ref{th2CJS}. $\Box$
    	\end{proof}
    	
    	\begin{remark}
	    	It is possible to prove Theorem~\ref{th1n} (also Lemma~\ref{lemma4n}) without using of clans.
    	\end{remark}
   
    	\textbf{Second proof of Lemma~\ref{lemma4n}}. We consider $[t)=\{x\in B:\ t\leq x\}$ and $(u]=\{x\in B:\ x\leq u\}$. It can be easily verified that $(u]$ is an ideal and that $[t)$ is a dual ideal. Suppose for the sake of contradiction that $(u]\cap [t)\neq\emptyset$, i.e. there is $x\in(u]\cap [t)$. We have $t\leq x\leq u$ and hence $t\leq u$ - a contradiction. Consequently $(u]\cap [t)=\emptyset$. By Lemma~\ref{lemma1n}, there exists a prime ideal $P$ of $\underline{B}$ with $(u]\subseteq P$ and $P\cap [t)=\emptyset$. Using Lemma~\ref{lemmastar}, we obtain that $U=B\setminus P$ is an abstract point of $\underline{B}$. Clearly $t\in U$ and $u\notin U$. $\Box$
    		
    	\bigskip
    	\textbf{Second proof of Theorem~\ref{th1n}}. The proof is the same as before with two differences. 
    	
    	For proving the reflexivity of $R$ we use Lemma~\ref{lemma30n0}.
    	
    	We prove that $aCb$ implies $h(a)Ch(b)$ in a similar way as in \cite{iv} (Lemma 3.8 (i)). Let $aCb$. We consider $P=\{x:\ x\overline{C}b\}$. We will prove that $P$ is an ideal. It suffices to show that $x+y\in P$ iff $x$, $y\in P$. Let $x+y\in P$. Consequently $(x+y)\overline{C}b$. Suppose for the sake of contradiction that $x\notin P$ or $y\notin P$. Without loss of generality $x\notin P$ and hence $xCb$; so $(x+y)Cb$ - a contradiction. Consequently $x$, $y\in P$. Now let $x$, $y\in P$ and suppose for the sake of contradiction that $x+y\notin P$. Consequently $(x+y)Cb$ and hence $bCx$ or $bCy$; so $x\notin P$ or $y\notin P$ - a contradiction. Consequently $x+y\in P$. Thus $P$ is an ideal. 
    	
    	We have also that $[a)$ is a dual ideal and $[a)\cap P=\emptyset$; so by Lemma~\ref{lemma1n}, there exists a prime ideal $P'$ of $\underline{B}$ with $P\subseteq P'$ and $P'\cap [a)=\emptyset$. By Lemma~\ref{lemmastar}, $F=B\setminus P'$ is an abstract point. 
    	
    	We consider $I=\{x:\ (\exists y\in F)(x\overline{C}y)\}$. We will prove that $I$ is an ideal. Let $x$, $y\in B$. It can be easily seen that $x+y\in I$ implies $x$, $y\in I$. Now let $x$, $y\in I$. We will prove $x+y\in I$. We have that $(\exists z_1\in F)(x\overline{C}z_1)$ and $(\exists z_2\in F)(y\overline{C}z_2)$. Since $z_1$, $z_2\in F$ and $F$ is an abstract point, there is a lower bound of $\{z_1,z_2\}$ $z$ such that $z\in F$. Suppose for the sake of contradiction that $(x+y)Cz$. Consequently $zCx$ or $zCy$. Without loss of generality $zCx$ but $z\leq z_1$; so $xCz_1$ - a contradiction. Consequently $(x+y)\overline{C}z$ and hence $x+y\in I$. Consequently $I$ is an ideal. Suppose for the sake of contradiction that there is $x\in [b)\cap I$. We have $(\exists y\in F)(b\leq x\overline{C}y)$. Since $y\in F$, $y\notin P$; so $yCb$; so $yCx$ - a contradiction. Consequently $[b)\cap I=\emptyset$. We have also that $[b)$ is a dual ideal, $I$ is an ideal; so by Lemma~\ref{lemma1n}, there is a prime ideal $I'$ with $I\subseteq I'$ and $I'\cap [b)=\emptyset$. By Lemma~\ref{lemmastar}, $F_1=B\setminus I'$ is an abstract point.
    	
    	It remains to prove that there are $U\in h(a)$, $V\in h(b)$ such that $(\forall x\in U)(\forall y\in V)(xCy)$. Clearly $F\in h(a)$ and $F_1\in h(b)$. Let $x\in F$, $y\in F_1$. Suppose for the sake of contradiction that $y\overline{C}x$. Consequently $y\in I$ and hence $y\in I'$; so $y\notin F_1$ - a contradiction. Consequently $yCx$. Thus $h(a)Ch(b)$. $\Box$

    	\section{A quantifier-free logic}
    	
    	We consider a quantifier-free language $\mathcal{L}$ which has
    	\begin{itemize}
    		\item constants: $0$, $1$;
    		\item functional symbols: $+$;
    		\item predicate symbols: $\leq$, $C$.
    	\end{itemize}

        We consider a quantifier-free logic $L$ which has axioms these of CJS and an only rule of inference - modus ponens. 
    	
    	\begin{theorem}[Completeness theorem]
    		Let $\varphi$ be a formula in $\mathcal{L}$. Then the following conditions are equivalent:\\
    		1) $\varphi$ is a theorem of $L$;\\
    		2) $\varphi$ is true in all topological contact algebras;\\
    		3) $\varphi$ is true in all DCJS;\\
    		4) $\varphi$ is true in all CJS;\\
    		5) $\varphi$ is true in all finite CJS with number of the elements $\leq 2^n+1$, where $n$ is the number of the variables of $\varphi$. 
    	\end{theorem}
    	\begin{proof}
    		Let $T$ be the set of the axioms of $L$. Condition 1) is equivalent to 1') $T\vdash\varphi$. By the well known Completeness theorem, Condition 1') is equivalent to 1'') $T\models\varphi$.
    		
    		1'')$\rightarrow$ 2) It can be easily verified.
    		
    		2)$\rightarrow$ 3) Let $\mathcal{A}$ be a DCJS and $v$ be a valuation in $\mathcal{A}$. We will prove that $(\mathcal{A},v)\models\varphi$. By Theorem~\ref{th2n}, there is a topological space $X$ and an isomorphic embedding $h$ of $\mathcal{A}$ in $\underline{RC(X)}$. Let the variables of $\varphi$ be $p_1,\ldots,p_n$, $n\geq 0$. We define a valuation $v_1$ in $\underline{RC(X)}$ in the following way:
    		\[v_1(p) =\left\{
    		\begin{array}{lll}
    		h(v(p)) & \quad &\textmd{if }p=p_1\textmd{ or }p=p_2\textmd{ or }\ldots\textmd{ or }p=p_n\\
    		\emptyset & \quad &\textmd{otherwise}
    		\end{array} \right.\]
    		Clearly $(\mathcal{A},v)\models\varphi$ iff $(\underline{RC(X)},v_1)\models\varphi$. Using 2), we get that $(\mathcal{A},v)\models\varphi$. 
    		
    		3)$\rightarrow$ 4) Let $\mathcal{A}$ be a CJS and $v$ be a valuation in $\mathcal{A}$. We will prove that $(\mathcal{A},v)\models\varphi$. By Theorem~\ref{th2CJS}, there is a topological space $X$ and an isomorphic embedding $h$ of $\mathcal{A}$ in $\underline{RC(X)}$. We define a valuation $v_1$ in $\underline{RC(X)}$ as above and we have $(\mathcal{A},v)\models\varphi$ iff $(\underline{RC(X)},v_1)\models\varphi$. By Proposition~\ref{pr1nn}, $\underline{RC(X)}$ is a DCJS and by 3), $(\underline{RC(X)},v_1)\models\varphi$. 
    		
    		4)$\rightarrow$ 5) Obviously.
    		
    		5)$\rightarrow$ 1'') Let $\mathcal{A}\models T$, i.e. $\mathcal{A}$ is a CJS. Let $v$ be a valuation in $\mathcal{A}$. We will prove that $(\mathcal{A},v)\models\varphi$. Let the variables of $\varphi$ be $p_1,\ldots,p_n$. We consider the set $S=\{v(p_{i_1})+\ldots+v(p_{i_m}):\ i_1<\ldots<i_m\leq n,\ m\geq 1\}\cup\{0,1\}$. Clearly $|S|\leq 2^n+1$. The structure $\mathcal{S}$ with universe $S$ is a substructure of $\mathcal{A}$ and since $\mathcal{A}$ is a CJS and the axioms of CJS can be considered as universal formulas, $\mathcal{S}$ is a CJS. We define a valuation $v_1$ in $\mathcal{S}$ in the following way:
    		\[v_1(p) =\left\{
    		\begin{array}{lll}
    		v(p) & \quad &\textmd{if }p=p_1\textmd{ or }p=p_2\textmd{ or }\ldots\textmd{ or }p=p_n\\
    		0 & \quad &\textmd{otherwise}
    		\end{array} \right.\]
    		By 5), $(\mathcal{S},v_1)\models\varphi$ and hence $(\mathcal{A},v)\models\varphi$. $\Box$
    	\end{proof}
    	
    	\begin{corollary}
    		$L$ is decidable.
    	\end{corollary}
    	
    	\section{Conclusion}
    	
    	Some possible future research directions are for example:
    	
    	\begin{itemize}
    	\item the complexity of the considered logic;
    	
    	\item is the theory of CJS finitely axiomatizable or not; is it possible axioms $A^1_{m,i}$ and $A_{n,i}$ to be simplified;
    	
    	\item to be obtained representations in $T_{1}$ and $T_{2}$ topological spaces by considering axiomatic extensions of CJS and DCJS;
    	
    	\item the language to be extended by considering as nondefinable primitives of the relations non-tangential inclusion and dual contact.
    	\end{itemize}
    	
    	\bigskip
    	\noindent{\bf Acknowledgements.} This paper is supported by National program "Young scientists and Postdoctoral candidates" 2020 of Ministry of Education and Science of Bulgaria.

\end{document}